\documentclass[a4paper,11pt]{article}
\usepackage[latin1]{inputenc}
\usepackage[english]{babel}
\usepackage{amsmath}
\usepackage{amsfonts}
\usepackage{amssymb}
\usepackage{epsfig}
\usepackage{amsopn}
\usepackage{amsthm}
\usepackage{color}
\usepackage{graphicx}
\usepackage{subfigure}
\usepackage{enumerate}
\newtheorem{theorem}{Theorem}[section]

\newtheorem{lemma}[theorem]{Lemma}
\newtheorem{proposition}[theorem]{Proposition}
\newtheorem{definition}[theorem]{Definition}

\newtheorem*{theorem*}{Theorem}
\newtheorem*{lemma*}{Lemma}
\newtheorem*{remark*}{Remark}
\newtheorem*{definition*}{Definition}
\newtheorem*{proposition*}{Proposition}
\newtheorem*{corollary*}{Corollary}
\numberwithin{equation}{section}
\newcommand{\real}{\mathbb{R}}

\let\ced=\c         

\def\qed{\,\unskip\kern 6pt \penalty 500
\raise -2pt\hbox{\vrule \vbox to8pt{\hrule width 6pt
\vfill\hrule}\vrule}\par}
\definecolor{darkblue}{rgb}{0.05, .05, .65}
\definecolor{darkgreen}{rgb}{0.1, .65, .1}
\definecolor{darkred}{rgb}{0.8,0,0}
\newcommand{\beqn}{\begin{equation}}
\newcommand{\eeqn}{\end{equation}}
\newcommand{\bear}{\begin{eqnarray}}
\newcommand{\eear}{\end{eqnarray}}
\newcommand{\bean}{\begin{eqnarray*}}
\newcommand{\eean}{\end{eqnarray*}}
\begin{document}

\title{\huge \bf Large time behavior for a quasilinear diffusion equation with weighted source}

\author{
\Large Razvan Gabriel Iagar\,\footnote{Departamento de Matem\'{a}tica
Aplicada, Ciencia e Ingenieria de Materiales y Tecnologia
Electr\'onica, Universidad Rey Juan Carlos, M\'{o}stoles,
28933, Madrid, Spain, \textit{e-mail:} razvan.iagar@urjc.es},\\
[4pt] \Large Marta Latorre\,\footnote{Departamento de Matem\'{a}tica
Aplicada, Ciencia e Ingenieria de Materiales y Tecnologia
Electr\'onica, Universidad Rey Juan Carlos, M\'{o}stoles,
28933, Madrid, Spain, \textit{e-mail:} marta.latorre@urjc.es},
\\[4pt] \Large Ariel S\'{a}nchez,\footnote{Departamento de Matem\'{a}tica
Aplicada, Ciencia e Ingenieria de Materiales y Tecnologia
Electr\'onica, Universidad Rey Juan Carlos, M\'{o}stoles,
28933, Madrid, Spain, \textit{e-mail:} ariel.sanchez@urjc.es}\\
[4pt] }
\date{}
\maketitle

\begin{abstract}
Large time behavior of general solutions to a class of quasilinear diffusion equations with a weighted source term
$$
\partial_tu=\Delta u^m+\varrho(x)u^p, \quad (x,t)\in\real^N\times(0,\infty),
$$
with $m>1$, $1<p<m$ and suitable functions $\varrho(x)$, is established. More precisely, we consider functions $\varrho\in C(\real^N)$ such that
$$
\lim\limits_{|x|\to\infty}(1+|x|)^{-\sigma}\varrho(x)=A\in(0,\infty),
$$
with $\sigma\in(\max\{-N,-2\},0)$ such that $L:=\sigma(m-1)+2(p-1)<0$. We show that, for all these choices of $\varrho$, solutions with initial conditions $u_0\in C(\real^N)\cap L^{\infty}(\real^N)\cap L^r(\real^N)$ for some $r\in[1,\infty)$ are global in time and, if $u_0$ is compactly supported, present the asymptotic behavior
$$
\lim\limits_{t\to\infty}t^{-\alpha}\|u(t)-V_*(t)\|_{\infty}=0,
$$
where $V_*$ is a suitably rescaled version of the unique compactly supported self-similar solution to the equation with the singular weight $\varrho(x)=|x|^{\sigma}$:
$$
U_*(x,t)=t^{\alpha}f_*(|x|t^{-\beta}), \qquad \alpha=-\frac{\sigma+2}{L}, \quad \beta=-\frac{m-p}{L}.
$$
This behavior is an interesting example of \emph{asymptotic simplification} for the equation with a regular weight $\varrho(x)$ towards the singular one as $t\to\infty$.
\end{abstract}

\


\smallskip


\section{Introduction}

In this paper we study the large time behavior of solutions to the Cauchy problem associated to the following quasilinear diffusion equation involving a spatially inhomogeneous source term
\begin{equation}\label{eq.gen}
u_t=\Delta u^m+\varrho(x)u^p, \quad (x,t)\in\real^N\times(0,\infty),
\end{equation}
in dimension $N\geq1$ and with exponents $m>1$, $1<p<m$, coupled with an initial condition
\begin{equation}\label{ic}
u(x,0)=u_0(x), \quad x\in\real^N,
\end{equation}
with
\begin{equation}\label{icond}
u_0\in L^{\infty}(\real^N)\cap C(\real^N)\cap L^r(\real^N), \quad u_0\geq0, \quad u_0\not\equiv0,
\end{equation}
for some $r\in(1,\infty)$. The weight $\varrho(x)$ is assumed to be a function satisfying the following conditions
\begin{equation}\label{cond.varrho}
\varrho(x)>0, \quad \varrho\in C(\real^N), \quad \lim\limits_{|x|\to\infty}(1+|x|)^{-\sigma}\varrho(x)=A\in(0,\infty),
\end{equation}
where $\sigma$ satisfies the following condition (which is compatible since $p\in(1,m)$)
\begin{equation}\label{cond.sigma}
\max\{-N,-2\}<\sigma<\sigma_*, \quad \sigma_*:=-\frac{2(p-1)}{m-1}\in(-2,0).
\end{equation}
A significant particular case on which we will give the details of the main technical parts of the analysis is the weight $\varrho(x)=(1+|x|)^{\sigma}$, with $\sigma$ satisfying \eqref{cond.sigma}, and leading to the equation
\begin{equation}\label{eq.reg}
u_t=\Delta u^m+(1+|x|)^{\sigma}u^p, \quad (x,t)\in\real^N\times(0,\infty).
\end{equation}
The main feature of equations such as Eq. \eqref{eq.gen} is the competition between the quasilinear diffusion term and the source term, which is weighted by a function depending on the space variable, pondering thus the strength of the reaction term according to its value at points $x\in\real^N$. It is a well-known fact, along the evolution, the diffusion term preserves the $L^1$-norm of a solution, while the source term increases it. Thus, in some cases, this growth of the $L^1$-norm might produce the phenomenon of finite time blow-up (that is, sudden unboundedness of $x\mapsto u(x,T)$ at some time $T\in(0,\infty)$) of the solutions.

Equations in the form \eqref{eq.gen} have been considered in \cite{BK87, Pi97, Pi98, SU24} for the semilinear case $m=1$, to quote but a few references, and \cite{Su02, SU21} for the quasilinear case $m>1$, restricted to the range $p>m$. Let us specially mention here the seminal paper \cite{AdB91} performing a rather deep study of the basic functional analytic theory of Eq. \eqref{eq.reg}, that is, with $\varrho(x)=(1+|x|)^{\sigma}$. Sharp conditions for local existence and non-existence, together with uniqueness, are established therein and, in fact, the analysis performed in \cite{AdB91} became our motivation for studying the dynamical properties of solutions to Eq. \eqref{eq.reg}, once the local well-posedness had been established. A different special case is taking $\varrho(x)$ as a compactly supported function, a problem addressed in works such as \cite{FdPV06, Liang12, FdP18, FdP22}. It has been seen that there exists $p_0>1$ such that for $p\in(1,p_0)$, any solution is global in time and grows up as $t\to\infty$, in striking contrast with the behavior of the standard reaction-diffusion equation (that is, $\varrho(x)=1$) where any non-trivial, non-negative solution blows up in finite time for $p\in(1,m)$, see \cite[Chapter~IV, Section 3.1]{S4}.

A special case in the general class \eqref{eq.gen}, which is not considered in this paper but it has a strong influence on the results, is represented by equations where $\varrho(x)=|x|^{\sigma}$ with $\sigma\in\real$, that is,
\begin{equation}\label{eq.sing}
u_t=\Delta u^m+|x|^{\sigma}u^p, \quad (x,t)\in\real^N\times(0,\infty),
\end{equation}
Since, on the one hand, the elliptic counterpart of \eqref{eq.sing} has been proposed in \cite[Equation (A.6)]{He73} (and integrated numerically therein) in a model in astrophysics for the stability of stellar systems, while, on the other hand, the linear parabolic equation with singular potential $|x|^{-2}$ has been studied in \cite{BG84} leading to rather striking results of existence or non-existence strongly related to the Hardy inequality and its best constant, equations of the form \eqref{eq.sing} have been more recently referred as \emph{Hardy-H\'enon equation}. In the semilinear case $m=1$ (and $p>1$), the mathematical analysis of the Hardy-H\'enon equations has been developed strongly in recent years, see for example \cite{BSTW17, T20, MS21, HT21, HS21, CIT21, CIT22, CITT24} and references therein. With respect to the quasilinear diffusion $m>1$, a Fujita-type exponent is identified in \cite{Qi98}, while some blow-up rates are established in \cite{AT05}, all them in the range $p>m$.

The study of Eq. \eqref{eq.sing} in the range $1<p<m$, which has no analogous for the semilinear Hardy-H\'enon equation, is very challenging from some points of view. As a result of a larger project of understanding the dynamical properties of solutions to Eq. \eqref{eq.sing}, the authors classified in recent works the solutions in radially symmetric self-similar form to this equation. It has been thus noticed that the sign of the constant
\begin{equation}\label{constL}
L:=\sigma(m-1)+2(p-1)
\end{equation}
plays a fundamental role in the form of such self-similar solutions. Indeed, the range $L<0$ (that is, $\sigma<\sigma_*$) has been analyzed in \cite{IMS23} and self-similar solutions in forward form, global in time, have been constructed. In particular, a very unexpected case of \emph{strong non-uniqueness} of solutions stemming from the trivial initial condition $u_0\equiv0$ came into light in this range, which corresponds exactly to \eqref{cond.sigma}, thus showing that supplementary conditions are needed for well-posedness. In the range $L>0$ (that is, $\sigma>\sigma_*$), the analysis performed in works such as \cite{IS21, ILS24b} shows that typical self-similar solutions present finite time blow-up, while the case $L=0$, that is, $\sigma=\sigma_*$, leads to exponential self-similar solutions (see \cite{ILS24a}).

Based on these observations, the qualitative theory of the quasilinear Hardy-H\'enon equation \eqref{eq.sing} has been developed in \cite{IL24} under the condition $\max\{-2,-N\}<\sigma<0$. In particular, regarding the range $L<0$, which is equivalent to \eqref{cond.sigma}, global existence for data as in \eqref{icond} follows from \cite[Theorem 2.2]{IL24}, while the uniqueness and the comparison principle for solutions to Eq. \eqref{eq.sing} hold true provided that the initial condition satisfies $u_0(x)>0$ for $x\in B(0,\delta)$, for some $\delta>0$, see \cite[Theorem 2.5]{IL24}. Finally, the large time behavior of solutions to Eq. \eqref{eq.sing} with compactly supported data and in the range $L<0$ is given in \cite[Theorem 2.11]{IL24}.

The emphasis of the present work relies on a surprising connection between solutions to Eq. \eqref{eq.gen} and a particular solution to Eq. \eqref{eq.sing} in the range \eqref{cond.sigma}. More precisely, we prove that, as $t\to\infty$, general solutions to the Cauchy problem \eqref{eq.gen}-\eqref{ic} approach a suitably rescaled form of the unique self-similar solution to Eq. \eqref{eq.sing} obtained in \cite{IMS23}. Let us stress here that the restriction \eqref{cond.sigma} is essential for this convergence to hold true. This fact leads to some natural open questions related to the complementary range $\sigma\leq-2$, in which Eq. \eqref{eq.gen} is well posed while Eq. \eqref{eq.sing} is likely to have no bounded solutions at all. After this discussion, we are in a position to state and discuss our main results of asymptotic convergence.

\medskip

\noindent \textbf{Main results.} Let us fix $m>1$, $p\in(1,m)$ and $\sigma$ as in \eqref{cond.sigma}, that is, $L<0$, where $L$ is defined in \eqref{constL}. Let us first recall here that, as established in \cite{IMS23}, there exists a unique radially symmetric self-similar solution to Eq. \eqref{eq.sing}
\begin{equation}\label{SSS}
U_*(x,t)=t^{\alpha}f_*(|x|t^{-\beta}), \quad \alpha=-\frac{\sigma+2}{L}>0, \quad \beta=-\frac{m-p}{L}>0,
\end{equation}
with a compactly supported and decreasing profile $f_*$ being a solution to the differential equation
\begin{equation}\label{SSODE}
(f^m)''(\xi)+\frac{N-1}{\xi}(f^m)'(\xi)-\alpha f(\xi)+\beta\xi f'(\xi)+\xi^{\sigma}f^p(\xi)=0, \quad \xi=|x|t^{-\beta}\in(0,\infty),
\end{equation}
and with initial conditions $f_*(0)>0$, $f_*'(0)=0$. This solution will play a fundamental role in the description of the large time behavior of solutions to Eq. \eqref{eq.gen}, as explained below. As a preliminary fact, it has been established in \cite[Theorem 2.11]{IL24} that any non-negative solution to Eq. \eqref{eq.sing} approaches $U_*$ as $t\to\infty$, in the self-similar scale of time.

Let us fix now $\varrho$ to be a continuous and bounded function satisfying \eqref{cond.varrho}. In order to state our main results, we first give, for the sake of completeness, the notion of weak solution that we consider throughout the paper, inspired from \cite[Definition 2.1]{Su02}.
\begin{definition}[Weak solution]
We say that $u$ is a weak solution to the Cauchy problem \eqref{eq.gen}-\eqref{ic} in $\real^N\times[0,T)$ if $u\geq0$ in $\real^N\times[0,T)$, $u\in C(\real^N\times[0,T))\cap L^{\infty}(\real^N\times(0,T))$ if for any $\tau\in(0,T)$ and any test function $\varphi\in C_0^{2,1}(\real^N\times[0,T))$, the following equality holds true:
\begin{equation}\label{weaksol}
\begin{split}
\int_{\real^N}&u(x,\tau)\varphi(x,\tau)\,dx-\int_{\real^N}u_0(x)\varphi(x,0)\,dx\\
&=\int_0^{\tau}\int_{\real^N}\left[u(x,t)\varphi_t(x,t)+u^m(x,t)\Delta\varphi(x,t)+\varrho(x)u^p(x,t)\varphi(x,t)\right]\,dx\,dt.
\end{split}
\end{equation}
We say that $u$ is a weak subsolution (respectively supersolution) to Eq. \eqref{eq.gen} if the equality sign in \eqref{weaksol} is replaced by $\leq$ (respectively $\geq$).
\end{definition}
Let us also mention here that we use throughout the paper the notation $u(t)$ to design the mapping $x\mapsto u(x,t)$ for $t>0$ fixed.

The first result establishes global existence and finite speed of propagation of solutions to \eqref{eq.gen}.
\begin{theorem}\label{th.fsp}
Let $u_0$ be as in \eqref{icond} and $\varrho$ be as in \eqref{cond.varrho}. Then, there exists a unique weak solution $u$ to the Cauchy problem \eqref{eq.gen}-\eqref{ic}, and this weak solution is global in time, that is, $u\in L^{\infty}(\real^N\times(0,\infty))$. Moreover, if $u_0$ is compactly supported, then $u(t)$ is compactly supported for any $t\in(0,\infty)$.
\end{theorem}
The uniqueness part of this statement is already well-known and we have kept it in the statement for the readers' convenience. Indeed, if $\varrho$ is as in \eqref{cond.varrho}, uniqueness and comparison principle follow from \cite[Proposition 2.2]{Su02} (a result which is in fact originally proved in \cite{BKP85}). Thus, we will only prove in this paper the global existence and the finite speed of propagation, following basically by comparison with a suitably delayed version of the self-similar solution $U_*$.

The main contribution of this paper is the following large time behavior result.
\begin{theorem}\label{th.asympt}
Let $u_0$ be as in \eqref{icond} and compactly supported and $\varrho$ be a function satisfying \eqref{cond.varrho}. Let $u$ be the (unique) solution to the Cauchy problem \eqref{eq.gen}-\eqref{ic}. Then we have
\begin{equation}\label{asympt}
\lim\limits_{t\to\infty}t^{-\alpha}\|u(t)-V_*(t)\|_{\infty}=0, \quad V_*(x,t)=A^{1/(m-p)}U_*\big(x,A^{(m-1)/(m-p)}t\big),
\end{equation}
where $U_*$ and $\alpha$ are the unique self-similar solution and the exponent introduced in \eqref{SSS} and $A\in(0,\infty)$ is the constant in \eqref{cond.varrho}.
\end{theorem}

Let us next discuss the most interesting points of Theorem \ref{th.asympt}. We observe that, for solutions to \eqref{eq.gen} with functions $\varrho$ satisfying \eqref{cond.varrho}, we establish in Theorem \ref{th.asympt} an interesting example of \emph{asymptotic simplification}. Indeed, we consider solutions to \eqref{eq.gen} with a bounded and continuous weight according to \eqref{cond.varrho} and we show that, as $t\to\infty$ and in the correct time scale, they approach a special solution to Eq. \eqref{eq.sing} which presents a coefficient featuring a singularity at $x=0$; for example, in the particular case $\varrho(x)=(1+|x|)^{\sigma}$ of Eq. \eqref{eq.reg}, we can just replace in the limit $(1+|x|)^{\sigma}$ by $|x|^{\sigma}$, despite the notable qualitative differences between these two weights. A similar asymptotic simplification has been discovered in the study of the non-homogeneous porous medium equation
\begin{equation}\label{nhpme}
\varrho(x)u_t=\Delta u^m, \quad \varrho(x)\sim|x|^{-\gamma}, \quad {\rm as} \ |x|\to\infty,
\end{equation}
in papers such as \cite{RV06, RV09} dealing with the range $\gamma\in(0,2)$, where the convergence towards self-similar profiles of the singular equation
$$
|x|^{-\gamma}u_t=\Delta u^m
$$
holds true uniformly in $\real^N$, and in \cite{KRV10} dealing with the range $\gamma>2$, where a similar convergence only holds true with further limitations for $\gamma$ and only in exterior sets of the form $|x|\geq\delta t^{\beta}$, $\delta>0$, $\beta>0$. However, Eq. \eqref{nhpme} appears to be of a different nature than Eq. \eqref{eq.gen}, thus it is an interesting fact that they share this special form of asymptotic simplification.

\medskip

\noindent \textbf{Open question.} The previous discussion is in our case limited to $\sigma$ satisfying \eqref{cond.sigma}, this limitation stemming from the existence of the self-similar solution $U_*$ to Eq. \eqref{eq.sing} presenting the correct time scale shared by Eq. \eqref{eq.gen}. Considering now $\sigma\leq-2$, Eq. \eqref{eq.gen} is still well-posed, since $\varrho\in L^{\infty}(\real^N)$ (and in particular for $\varrho(x)=(1+|x|)^{\sigma}$, the analysis in \cite{AdB91} covers the whole range $\sigma\in(-\infty,0)$ of negative exponents). However, it is likely that the equation with singular weight Eq. \eqref{eq.sing} has no bounded self-similar solutions (or even no bounded solutions at all), a fact established for $\sigma=-2$ in \cite{IS23}. Thus, letting $\sigma\leq-2$ in \eqref{cond.varrho}, the question of finding the profiles for the large time behavior of solutions seems to be a very interesting one.

\medskip

\noindent \textbf{Structure of the paper}. Apart from the Introduction, the paper is divided into two sections. The proof of both Theorem \ref{th.fsp} and \ref{th.asympt} for Eq. \eqref{eq.reg}, taking thus $\varrho(x)=(1+|x|)^{\sigma}$ as a significant example (whose functional analytic theory is available in \cite{AdB91}) on which it is easier to present the details of the proofs, is given in Section \ref{sec.reg}. As we shall see, the most difficult technical step in the proof is the construction of suitable subsolutions ensuring that the solutions do not collapse to zero as $t\to\infty$. The final Section \ref{sec.gen} contains the proof of Theorems \ref{th.fsp} and \ref{th.asympt} for general weights $\varrho$ satisfying \eqref{cond.varrho}, having as starting point the details of the proof given before for $\varrho(x)=(1+|x|)^{\sigma}$.

\section{Proofs of the main results with $\varrho(x)=(1+|x|)^{\sigma}$}\label{sec.reg}

Throughout this section we fix $\varrho(x)=(1+|x|)^{\sigma}$, that is, we work only with Eq. \eqref{eq.reg}, an equation whose (local in time) mathematical theory has been developed in \cite{AdB91}. The proof of Theorem \ref{th.fsp} follows partly from this theory and is given below.
\begin{proof}[Proof of Theorem \ref{th.fsp}]
Let $u_0$ be as in \eqref{icond}. Introducing the notation
$$
[[u_0]]:=\sup\limits_{R\geq1}R^{-2/(m-1)}\frac{1}{|B(0,R)|}\int_{B(0,R)}u_0(y)\,dy,
$$
and recalling that $\sigma<\sigma_*$ by \eqref{cond.sigma}, we are in the conditions of \cite[Theorem 3.1, p. 371]{AdB91} stemming that local existence of a solution to the Cauchy problem \eqref{eq.reg}-\eqref{ic} in $\real^N\times(0,T)$ is ensured for some $T>0$, provided that $[[u_0]]<\infty$. In our case, the latter holds true since $u_0\in L^{\infty}(\real^N)$, and we readily find that
$$
[[u_0]]\leq\sup\limits_{R\geq1}R^{-2/(m-1)}\|u_0\|_{\infty}=\|u_0\|_{\infty}.
$$
Moreover, the negativity of $\sigma$ entails that $(1+|x|)^{\sigma}\in L^{\infty}(\real^N)$ and the uniqueness and comparison principle follow from \cite[Proposition 2.2]{Su02}. Finally, let $\overline{u}$ be any solution to the Cauchy problem \eqref{eq.sing}-\eqref{ic}; indeed, such a solution exists globally according to \cite[Theorem 2.2]{IL24}, since $L<0$ and $u_0$ satisfies \eqref{icond}. We deduce from the negativity of $\sigma$ that
\begin{equation}\label{comp.sigma}
(1+|x|)^{\sigma}<|x|^{\sigma}, \quad x\in\real^N,
\end{equation}
hence $\overline{u}$ is a supersolution to \eqref{eq.reg} which is global in time. The comparison principle \cite[Proposition 2.2]{Su02} then entails that $u(t)\in L^{\infty}(\real^N)$ for any $t\in(0,\infty)$ and thus $u$ is a global solution.

It remains to prove the finite speed of propagation. Let thus $u_0$ be a compactly supported initial condition satisfying \eqref{icond}. Let
\begin{equation}\label{edge.supp}
\xi_0:=\sup\{\xi>0: f_*(\xi)>0\}\in(0,\infty)
\end{equation}
be the edge of the support of the profile $f_*$ of the compactly supported self-similar solution $U_*$ to Eq. \eqref{eq.sing} introduced in \eqref{SSS}, and
$$
R_0:=\sup\{|x|: x\in\real^N, u_0(x)>0\}\in(0,\infty)
$$
be the maximum amplitude of the positivity set of $u_0$. Then, for any $\tau>0$, the support of $U_*(\tau)$ is $B(0,\tau^{\beta}\xi_0)$. Noticing that
$$
\lim\limits_{\tau\to\infty}U_*(R_0,\tau)=\lim\limits_{\tau\to\infty}\tau^{\alpha}f_*(R_0\tau^{-\beta})=\infty, \quad \lim\limits_{\tau\to\infty}\tau^{\beta}\xi_0=\infty,
$$
we can choose $\tau_{\infty}>0$ such that
\begin{equation}\label{comp.above}
U_*(R_0,\tau_{\infty})\geq\|u_0\|_{\infty}, \quad \tau_{\infty}^{\beta}\xi_0>2R_0.
\end{equation}
Since $U_*$ is radially decreasing on its support, we readily infer from \eqref{comp.above} that $U_*(x,\tau_{\infty})\geq u_0(x)$ for any $x\in B(0,R_0)$ and, by the definition of $R_0$, that $U_*(\tau_{\infty})\geq u_0$ in $\real^N$. We next deduce from the inequality \eqref{comp.sigma} that the function $U_*(\cdot,\cdot+\tau_{\infty})$ with $\tau_{\infty}\in(0,\infty)$ constructed above is a supersolution to the Cauchy problem \eqref{eq.reg}-\eqref{ic}, and the comparison principle ensures that
$$
u(x,t)\leq U_*(x,t+\tau_{\infty}), \quad (x,t)\in\real^N\times(0,\infty),
$$
hence $u(t)$ remains compactly supported at any time $t>0$.
\end{proof}
We next go towards the proof of Theorem \ref{th.asympt} for solutions to Eq. \eqref{eq.reg}. As commented in the Introduction, this is a very interesting case of asymptotic simplification, and the main difficulty in the proof is to establish a uniform lower bound for $u(t)$, $t\in(0,\infty)$, matching with the same grow-up rate $t^{\alpha}$ as the claimed profile $U_*$ introduced in \eqref{SSS}. Observe that any solution to Eq. \eqref{eq.sing} is a \emph{strict supersolution} for Eq. \eqref{eq.reg}, thus we cannot directly employ such a solution to bound solutions to Eq. \eqref{eq.reg} from below. This fact reinforces the need to construct in a different (and more involved) way a subsolution to the Cauchy problem \eqref{eq.reg}-\eqref{ic} with the correct grow-up rate $t^{\alpha}$ as $t\to\infty$. This construction is more technical, and we need a number of preparatory lemmas for it. The first of them is related to Eq. \eqref{eq.sing}.
\begin{lemma}\label{lem.subs1}
If $v$ is a solution to \eqref{eq.sing} and $c>0$, then
\begin{equation*}
w_{c}(x,t):=\lambda_c^{1/(m-1)}v(x,\lambda_ct), \quad \lambda_c=c^{(m-1)/(m-p)}
\end{equation*}
is a solution to
\begin{equation}\label{eq.sing.c}
w_t=\Delta w^m+c|x|^{\sigma}w^p.
\end{equation}
\end{lemma}
\begin{proof}
We give a formal proof assuming that $v$ is a classical solution. By direct calculation, we have
\begin{equation*}
\begin{split}
&\partial_tw_c(x,t)=\lambda_c^{m/(m-1)}\partial_tv(x,\lambda_ct),\\
&\Delta w_c^m(x,t)=\lambda_c^{m/(m-1)}\Delta v^m(x,\lambda_ct),\\
&c|x|^{\sigma}w_c^p(x,t)=\lambda_c^{(m-p)/(m-1)}\lambda_c^{p/(m-1)}|x|^{\sigma}v^p(x,\lambda_ct)=\lambda_c^{m/(m-1)}|x|^{\sigma}v^p(x,\lambda_ct).
\end{split}
\end{equation*}
We thus find that
$$
\big(\partial_tw_c-\Delta w_c^m-c|x|^{\sigma}w_c^p\big)(x,t)=\lambda_c^{m/(m-1)}\big(\partial_tv-\Delta v^m-|x|^{\sigma}v^p\big)(x,\lambda_ct)=0,
$$
for any $(x,t)\in\real^N\times(0,\infty)$, which concludes the proof. If $v$ is only a weak solution, the proof is done in the same way on the weak formulation \eqref{weaksol} by changing variables in the integrals, and we omit here the details.
\end{proof}
\begin{lemma}\label{lem.subs2}
Let $w$ be a solution to \eqref{eq.sing.c} with $c\in(0,1)$. Then $w$ is a subsolution to Eq. \eqref{eq.reg} in the region
\begin{equation}\label{reg.c}
\left\{(x,t)\in\real^N\times(0,\infty): |x|\geq K(c):=\frac{1}{c^{1/\sigma}-1}\right\}.
\end{equation}
\end{lemma}
\begin{proof}
By direct calculation, we have
\begin{equation}\label{interm2}
w_t-\Delta w^m-(1+|x|)^{\sigma}w^p=\big[c|x|^{\sigma}-(1+|x|)^{\sigma}\big]w^p.
\end{equation}
Notice that the fact that $(x,t)$ belongs to the region \eqref{reg.c} implies that $|x|(c^{1/\sigma}-1)\geq1$, which is equivalent, taking into account the negativity of $\sigma$, to $c|x|^{\sigma}\leq (1+|x|)^{\sigma}$, whence the right hand side of \eqref{interm2} is non-positive, completing the proof.
\end{proof}
The next preparatory result is based on dynamical systems techniques developed in the previous paper \cite{IMS23}.
\begin{lemma}\label{lem.subs3}
There is at least one subsolution $\tilde{u}$ to Eq. \eqref{eq.sing} in the self-similar form
\begin{equation}\label{SSSub}
\tilde{u}(x,t)=t^{\alpha}f(|x|t^{-\beta}), \quad {\rm supp}\,f=[R_1,R_2], \quad R_2>R_1>0,
\end{equation}
with $\alpha$, $\beta$ defined in \eqref{SSS} and the function $f$ is a profile solving, on its positivity set, the differential equation \eqref{SSODE}.
\end{lemma}
\begin{proof}
We look for solutions (at least at points where $\tilde{u}>0$) to Eq. \eqref{eq.sing} in the form \eqref{SSSub}. It is a straightforward calculation to check that the profile $f$ satisfies on its positivity set the differential equation \eqref{SSODE}. The rest of the proof is based on the change of variable employed for the classification of self-similar solutions in \cite{IMS23}, together with a phase-plane analysis technique used therein, and we briefly recall it here for the sake of completeness. We set
\begin{equation}\label{PSchange}
X(\eta)=\frac{m}{\alpha}\xi^{-2}f(\xi)^{m-1}, \ \ Y(\eta)=\frac{m}{\alpha}\xi^{-1}f(\xi)^{m-2}f'(\xi), \ \ Z(\eta)=\frac{1}{\alpha}\xi^{\sigma}f(\xi)^{p-1},
\end{equation}
where the new independent variable $\eta$ is defined implicitly the differential equation
\begin{equation}\label{indep}
\frac{d\eta}{d\xi}=\frac{\alpha}{m}\xi f(\xi)^{1-m}.
\end{equation}
We deduce by direct calculations that \eqref{SSODE} is transformed, in variables $(X,Y,Z,\eta)$, into the autonomous dynamical system
\begin{equation}\label{PSsyst}
\left\{\begin{array}{ll}\dot{X}=X[(m-1)Y-2X],\\
\dot{Y}=-Y^2-\frac{\beta}{\alpha}Y+X-NXY-XZ,\\
\dot{Z}=Z[(p-1)Y+\sigma X],\end{array}\right.
\end{equation}
where the dot derivatives are taken with respect to the independent variable $\eta$. Its equilibrium points, either finite or at infinity (using the compactification of the phase space to the equator of the Poincar\'e hypersphere according to the theory exposed, for example, in \cite[Section 3.10]{Pe}), have been classified and analyzed carefully in \cite[Section 2]{IMS23} (finite points) and \cite[Section 3]{IMS23} (points at infinity). Among the critical points at infinity, we are interested in the ones denoted by $Q_2$ and $Q_3$ (in the notation of \cite{IMS23}), which encode the limits $Y\to\pm\infty$ in the phase space; that is, in a formal writing, the points containing the behavior of trajectories of the system \eqref{PSsyst} such that $(X,Y,Z)\to(0,\pm\infty,0)$ along them. Their analysis, performed in \cite[Lemma 3.4]{IMS23} following the details in \cite[Lemma 2.6]{IS21}, established that

$\bullet$ the critical point $Q_2$ corresponding to the limit $(X,Y,Z)\to(0,\infty,0)$ is an unstable node, and undoing the transformation \eqref{PSchange} on the trajectories of the system \eqref{PSsyst} going out of it, they are mapped into profiles $f$ such that there are $R_1\in(0,\infty)$ and $\delta>0$ with
\begin{equation}\label{interm3}
f(R_1)=0, \quad f(\xi)>0 \ {\rm for} \ \xi\in(R_1,R_1+\delta), \quad (f^m)'(R_1)>0.
\end{equation}

$\bullet$ the critical point $Q_3$ corresponding to the limit $(X,Y,Z)\to(0,-\infty,0)$ is a stable node, and undoing the transformation \eqref{PSchange} on the trajectories of the system \eqref{PSsyst} entering it, they are mapped into profiles $f$ such that there are $R_2\in(0,\infty)$ and $\delta\in(0,R_2)$ with
\begin{equation}\label{interm4}
f(R_2)=0, \quad f(\xi)>0 \ {\rm for} \ \xi\in(R_2-\delta,R_2), \quad (f^m)'(R_2)<0.
\end{equation}

We thus notice that a trajectory of the system \eqref{PSsyst} connecting $Q_2$ and $Q_3$ in the phase space associated to the system \eqref{PSsyst} corresponds to a profile $f$ matching both behaviors \eqref{interm3} and \eqref{interm4} at $\xi=R_1$, respectively at $\xi=R_2$. Hence ${\rm supp}\,f=[R_1,R_2]$ and we have our subsolution $\tilde{u}$ defined by \eqref{SSSub}. Observe that $\tilde{u}$ is in fact a solution to Eq. \eqref{eq.sing} on $(R_1,R_2)$ and a subsolution at its zeros $\xi=R_1$ and $\xi=R_2$ if extended by zero on both sides outside the interval $[R_1,R_2]$.

We are thus left to prove the existence of a trajectory connecting these critical points $Q_2$ and $Q_3$. To this end, we first perform a further change of variable in order to replace $Z$ by $W=XZ$ in the system \eqref{PSsyst}, corresponding in terms of profiles to
$$
W(\eta)=\frac{m}{\alpha^2}\xi^{\sigma-2}f(\xi)^{m+p-2},
$$
with $\eta$ introduced in \eqref{indep}. It is straightforward to check that the system \eqref{PSsyst} is transformed into
\begin{equation}\label{PSsystbis}
\left\{\begin{array}{ll}\dot{X}=X[(m-1)Y-2X],\\
\dot{Y}=-Y^2-\frac{\beta}{\alpha}Y+X-NXY-W,\\
\dot{W}=W[(m+p-2)Y+(\sigma-2)X],\end{array}\right.
\end{equation}
and the analysis of the critical points $Q_2$ and $Q_3$ remains unchanged. We establish next a trajectory connecting $Q_2$ to $Q_3$ in the invariant plane $\{X=0\}$, that is, in the phase plane associated to the reduced system
\begin{equation}\label{PSplane}
\left\{\begin{array}{ll}
\dot{Y}=-Y^2-\frac{\beta}{\alpha}Y-W,\\
\dot{W}=(m+p-2)YW,\end{array}\right.
\end{equation}
Observe that, in this system, the critical point $P_1=(-\beta/\alpha,0)$ is a saddle point. Indeed, the linearization of the system \eqref{PSplane} in a neighborhood of this point has the matrix
$$
M(P_1)=\left(
         \begin{array}{cc}
           \frac{\beta}{\alpha} & -1 \\[1mm]
           0 & -\frac{(m+p-2)\beta}{\alpha} \\
         \end{array}
       \right),
$$
with a one-dimensional unstable manifold included in the invariant $Y$-axis (as indicated by the eigenvector $e_1=(1,0)$) and a one-dimensional stable manifold $W_s(P_1)$ consisting in a single trajectory tangent to the eigenvector
$$
e_2=\left(1,\frac{(m+p-1)\beta}{\alpha}\right).
$$
We prove next that the unique trajectory in $W_s(P_1)$ arrives from the critical point $Q_2$. Consider the isocline $\dot{Y}=0$ of the system \eqref{PSplane}, that is,
\begin{equation}\label{isoY}
Y^2+\frac{\beta}{\alpha}Y+W=0, \quad {\rm with \ normal \ vector} \quad \overline{n}=\left(2Y+\frac{\beta}{\alpha},1\right).
\end{equation}
On the one hand, the direction of the flow of the system \eqref{PSplane} across the isocline \eqref{isoY} is given by the sign of the expression
$$
(m+p-2)\left(-Y^2-\frac{\beta}{\alpha}Y\right)Y\leq0, \quad Y\in\left[-\frac{\beta}{\alpha},0\right].
$$
On the other hand, the slope of the isocline \eqref{isoY} near $P_1$ is given by
$$
\frac{dW}{dY}\Big|_{Y=-\beta/\alpha}=\frac{\beta}{\alpha}<\frac{(m+p-1)\beta}{\alpha},
$$
the latter being the slope of the vector $e_2$ indicating the direction of the unique trajectory contained in $W_s(P_1)$. We infer from the direction of the flow across \eqref{isoY} that the trajectory contained in $W_s(P_1)$ lies forever in the region $\dot{Y}<0$. The inverse function theorem allows us then to parametrize this trajectory as a curve $W=W_1(Y)$, with derivative
\begin{equation}\label{interm9}
W_1'(Y)=-\frac{(m+p-2)W_1(Y)Y}{Y^2+(\beta/\alpha)Y+W_1(Y)},
\end{equation}
and it is easy to see that this curve cannot have a vertical asymptote $W_1(Y)\to\infty$ as $Y\to Y_0$ for some $Y_0\in(-\beta/\alpha,\infty)$; indeed, in such case, we would deduce from \eqref{interm9} that in a left neighborhood of such a point $Y_0$ we have
$$
|W_1'(Y)|\leq(m+p-2)|Y|
$$
thus the growth would be at most linear, contradicting the existence of the vertical asymptote. It then follows that the trajectory in $W_s(P_1)$ arrives from the half-plane $\{Y>0\}$, in which $W_1(Y)$ is decreasing with $Y$. Arguing by contradiction, we infer from an easy application of the Poincar\'e-Bendixon Theorem \cite[Section 3.7]{Pe} together with the non-existence of any finite critical point in the half-plane $\{Y>0\}$ that the unique trajectory in $W_s(P_1)$ arrives from $Q_2$. Since $Q_2$ is an unstable node and the trajectory $W_1(Y)$ is a separatrix, any trajectory going out of it into the region $\{W>W_1(Y)\}$ remains forever in the region $W>W_1(Y)$ for $Y\geq-\beta/\alpha$ and cannot enter $P_1$. Observe also that a blow-up of such trajectories as $Y\to Y_1$ for some $Y_1\in(0,\infty)$ is ruled out again by \eqref{interm9}. An easy argument by contradiction based on the monotonicity of both $W$ and $Y$ along such a trajectory, the non-existence of finite critical points in the region $\{-\infty<Y<-\beta/\alpha\}$ and the Poincar\'e-Bendixon Theorem gives that all these trajectories starting from $Q_2$ in the region $\{W>W_1(Y)\}$ must enter the stable node $Q_3$. The stability of $Q_2$ and $Q_3$ entails that there are other orbits of the phase space associated to the system \eqref{PSsystbis} not contained in the plane $\{X=0\}$ connecting them, as claimed. 
\end{proof}
Since the previous arguments can be easily followed with the help of a picture, we plot in Figure \ref{fig1} a typical phase portrait associated to the system \eqref{PSplane} according to the proof of Lemma \ref{lem.subs3}. Observe that the critical point $P_0=(0,0)$ appearing in the figure is not analyzed in the proof, as it is not needed for our goals.

\begin{figure}[ht!]
  \begin{center}
  \includegraphics[width=11cm,height=8cm]{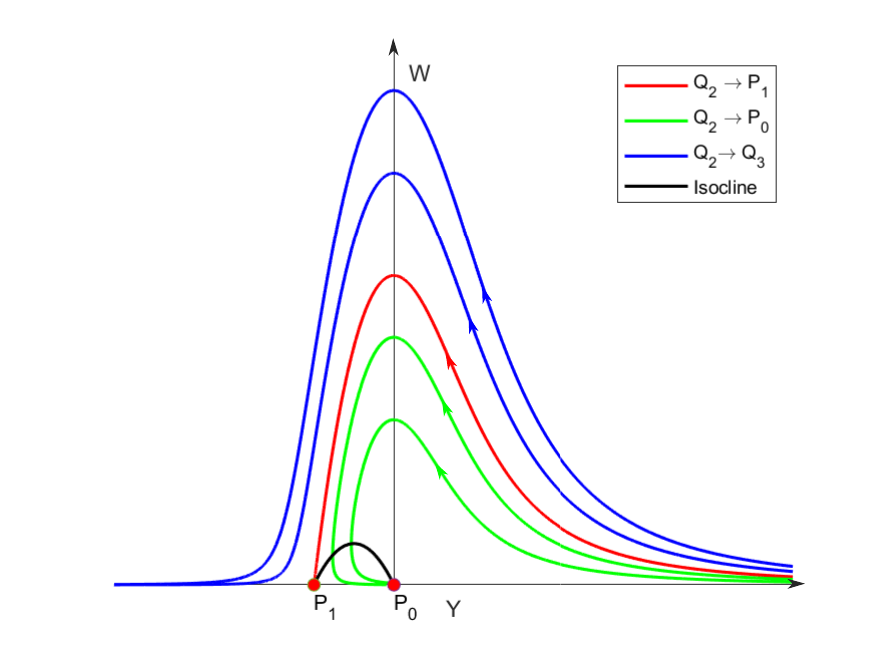}
  \end{center}
  \caption{The phase plane associated to the system \eqref{PSplane}. Experiment for $m=3$, $p=2$, $N=4$, $\sigma=-1.5$}\label{fig1}
\end{figure}

The last preparatory result stems from a very well known property of solutions to the porous medium equation, and we state and prove it here for the sake of completeness.
\begin{lemma}\label{lem.subs4}
Let $u_0$ be as in \eqref{icond} and $M>0$. Then, there exists $t_0>0$ such that the solution $u$ to the Cauchy problem \eqref{eq.reg}-\eqref{ic} satisfies $B(0,M)\subset{\rm supp}\,u(t_0)$.
\end{lemma}
\begin{proof}
Let $w$ be the solution to the Cauchy problem
$$
w_t=\Delta w^m, \ {\rm in} \ \real^N\times(0,\infty), \quad w(x,0)=u_0(x), \ {\rm for} \ x\in\real^N.
$$
Standard results in the theory of the porous medium equation (see for example \cite[Proposition 9.19]{VPME}) show that there is $t_0>0$ such that $w(x,t_0)>0$ for any $x\in B(0,M)$. Since $u$ is a supersolution for the porous medium equation, the comparison principle then gives that $u(x,t_0)\geq w(x,t_0)>0$ in $B(0,M)$, completing the proof.
\end{proof}
We are now able to construct a subsolution to Eq. \eqref{eq.reg} that will be decisive in the proof of Theorem \ref{th.asympt} by preventing that the limit of $u(t)$ as $t\to\infty$, in the right time scale, can vanish identically.
\begin{proposition}\label{prop.subs}
Let $u_0$ be as in \eqref{icond} and $u$ be the solution to the Cauchy problem \eqref{eq.reg}-\eqref{ic}. Then there exist $\lambda_*\in(0,1)$, a subsolution $W_*$ to Eq. \eqref{eq.reg} in the self-similar form \eqref{SSSub} with a profile $f$ such that ${\rm supp}\,f=[R_1,R_2]\subset(0,\infty)$ and a time $t_0>0$ such that $W_*(x,t)\leq u(x,t+t_0)$ for any $t\geq1/\lambda_*$.
\end{proposition}
\begin{proof}
Let $v$ be a self-similar subsolution to Eq. \eqref{eq.sing} in the form \eqref{SSSub} with a profile $f$ such that ${\rm supp}\,f=[R_1,R_2]\subset(0,\infty)$, as given by Lemma \ref{lem.subs3}. We infer from Lemma \ref{lem.subs4} that there is $t_0>0$ such that $B(0,R_2)\subset{\rm supp}\,u(t_0)$. We introduce the following extremal values:
$$
h:=\min\limits_{x\in B(0,R_2)}u(x,t_0)>0, \quad H:=\max\limits_{x\in\real^N}v(x,1)=\max\limits_{\xi\in[R_1,R_2]}f(\xi)
$$
and
\begin{equation}\label{interm5}
\lambda_*:=\min\left\{\left(\frac{h}{H}\right)^{m-1},\left(\frac{1+R_1}{R_1}\right)^{\sigma(m-1)/(m-p)}\right\}.
\end{equation}
Since $1<p<m$, the negativity of $\sigma$ entails that $\lambda_*\in(0,1)$. Let us define then
$$
W_*(x,t)=\lambda_*^{1/(m-1)}v(x,\lambda_*t),
$$
which is a subsolution (and in fact a solution at points where $W_*>0$, according to Lemma \ref{lem.subs1}) to the equation
$$
W_t=\Delta W^m+\lambda_*^{(m-p)/(m-1)}|x|^{\sigma}W^p.
$$
We then deduce from \eqref{interm5} that
\begin{equation}\label{interm6}
\max\limits_{x\in\real^N}W_*\left(x,\frac{1}{\lambda_*}\right)=\lambda_*^{1/(m-1)}\max\limits_{x\in\real^N}v(x,1)=\lambda_*^{1/(m-1)}H\leq h=\min\limits_{x\in B(0,R_2)}u(x,t_0).
\end{equation}
Moreover, since $\lambda_*\in(0,1)$, we infer from Lemma \ref{lem.subs2} that $W_*$ is a subsolution to Eq. \eqref{eq.reg} in the region
$$
\left\{(x,t)\in\real^N\times(0,\infty): |x|>K(\lambda_*^{(m-p)/(m-1)})=\frac{1}{\lambda_*^{(m-p)/\sigma(m-1)}-1}\right\}.
$$
Noticing that the definition \eqref{interm5} of $\lambda_*$ easily leads to $K(\lambda_*^{(m-p)/(m-1)})\leq R_1$, we infer that $W_*$ is a subsolution to Eq. \eqref{eq.reg} for $|x|\geq R_1$, $t>0$. Recalling that
$$
{\rm supp}\,W_*\left(\frac{1}{\lambda_*}\right)={\rm supp}\,v(1)={\rm supp}\,f=[R_1,R_2],
$$
the self-similar form of $W_*$ readily implies that $W_*(x,t)=0$ for $t>1/\lambda_*$ and $x\in B(0,R_1)$ (the lower edge of the support advances in the sense that its absolute value increases with time). We deduce from this fact and \eqref{interm6} that, on the one hand, $W_*(x,t)=0\leq u(x,t_0+t)$ for $t>1/\lambda_*$, $|x|=R_1$, and on the other hand, $W_*(x,1/\lambda_*)\leq u(x,t_0)$ for any $x\in\real^N$. The comparison principle then entails that
$$
W_*(x,t)\leq u\left(x,t+t_0-\frac{1}{\lambda_*}\right), \quad (x,t)\in\real^N\times\left(\frac{1}{\lambda_*},\infty\right),
$$
completing the proof.
\end{proof}
We are now in a position to advance towards the proof of the large time behavior \eqref{asympt} for solutions to Eq. \eqref{eq.reg}. To this end, pick $u$ to be the solution to Eq. \eqref{eq.reg} with initial condition $u_0$ as in \eqref{icond}. We next introduce the self-similar variables
\begin{equation}\label{SSvar}
v(y,s)=t^{-\alpha}u(x,t), \quad y=xt^{-\beta}, \quad s=\ln\,t, \quad (x,t)\in\real^N\times[1,\infty),
\end{equation}
and it follows by direct calculation that
\begin{equation*}
\begin{split}
&\partial_tu(x,t)=\alpha t^{\alpha-1}v(y,s)-\beta t^{\alpha-1}y\cdot\nabla v(y,s)+t^{\alpha-1}\partial_sv(y,s),\\
&\Delta u^m(x,t)=t^{m\alpha-2\beta}\Delta v^m(y,s),\\
&(1+|x|)^{\sigma}u^p(x,t)=t^{p\alpha}(1+t^{\beta}|y|)^{\sigma}v^p(y,s)=t^{p\alpha+\beta\sigma}(e^{-\beta s}+|y|)^{\sigma}v^p(y,s).
\end{split}
\end{equation*}
Recalling the values of $\alpha$ and $\beta$ in \eqref{SSS}, we readily deduce that
$$
\alpha-1=m\alpha-2\beta=p\alpha+\beta\sigma,
$$
whence we are left with the equation satisfied by $v$ in the $(y,s)$ variables, namely
\begin{equation}\label{eq.reg.SS}
\partial_sv=\Delta v^m+\beta y\cdot\nabla v-\alpha v+(e^{-\beta s}+|y|)^{\sigma}v^p.
\end{equation}
Since $t\to\infty$ translates into $s=\ln\,t\to\infty$, we observe from \eqref{eq.reg.SS} that the asymptotic simplification claimed in the statement of Theorem \ref{th.asympt} takes place at a formal level, due to the positivity of $\beta$. We make this convergence rigorous in the forthcoming lines, by employing as theoretical argument the stability theorem for dynamical systems known in literature as \emph{the S-theorem} and introduced by Galaktionov and V\'azquez in \cite{GV91} (see also the monograph \cite{GV03}). For the reader's convenience, before going into the details of the proof, we recall below the theoretical framework of the S-theorem following \cite{GV91}. In an abstract setting, we consider two dynamical systems given respectively by the \emph{limit equation} and the \emph{perturbed equation}
\begin{equation}\label{lim.eq}
w_t=W(w), \quad {\rm respectively} \quad v_t=V(t,v),
\end{equation}
and assume that the following three conditions are fulfilled:

$\bullet$ (H1) We consider a class $\mathcal{S}$ of solutions $v\in C([0,\infty),X)$ of the perturbed equation with values in a complete metric space $X$. Assume that $\{v(t)\}_{t>0}$ is relatively compact in $X$, for any $v\in\mathcal{S}$.

$\bullet$ (H2) Given a solution $v\in\mathcal{S}$ of the perturbed equation, and assuming that, for a sequence $\{t_j\}_{j\geq1}$ with $t_j\to\infty$, $v(t+t_j)$ converges (in the topology of the metric space $X$) to a function $w(t)$ as $j\to\infty$, then $w$ is a solution to the limit equation.

$\bullet$ (H3) The $\omega$-limit set $\Omega$ of the limit equation in $X$, that is, the set of functions $f\in X$ such that there is a solution $w$ to the limit equation and a sequence $\{t_j\}_{j\geq1}$ with $t_j\to\infty$ and $w(t_j)\to f$ as $j\to\infty$ (in the topology of $X$) is nonvoid, compact, and uniformly stable in the sense that for any $\epsilon>0$, there is $\delta=\delta(\epsilon)>0$ such that for any solution $w$ to the limit equation with $d(w(0),\Omega)\leq\delta$, then $d(w(t),\Omega)\leq\epsilon$ for any $t>0$.

Under these circumstances, the $\omega$-limit sets of solutions to the perturbed equation $v_t=V(t,v)$ in the class $\mathcal{S}$ are contained in $\Omega$. In other words, the asymptotic limits of solutions to $v_t=V(t,v)$ are contained among the asymptotic limits of solutions to the limit equation $w_t=W(w)$. With these preliminaries, we can give the proof of Theorem \ref{th.asympt} for Eq. \eqref{eq.reg}.

\begin{proof}[Proof of Theorem \ref{th.asympt}]
Let $u$ be the solution to the Cauchy problem \eqref{eq.reg}-\eqref{ic} with $u_0$ as in the statement of Theorem \ref{th.asympt} and let $(v,y,s)$ defined in \eqref{SSvar}. Our goal is to show that the $\omega$-limit set as $s\to\infty$ of the solutions to the perturbed (non-autonomous with respect to the time variable) equation \eqref{eq.reg.SS} coincides with the $\omega$-limit set of the solutions to the autonomous (with respect to the time variable) limit equation
\begin{equation}\label{eq.sing.SS}
\partial_sv=\Delta v^m+\beta y\cdot\nabla v-\alpha v+|y|^{\sigma}v^p,
\end{equation}
which, by undoing the self-similar change of variable \eqref{SSvar}, readily leads to Eq. \eqref{eq.sing}. Before passing to check the hypothesis of application of the S-theorem, we translate into self-similar variables the upper and lower bounds obtained by comparison with the sub- and supersolutions we have already constructed.

On the one hand, as shown in the proof of Theorem \ref{th.fsp}, there is $\tau_{\infty}\in(0,\infty)$ such that $u(x,t)\leq U_{*}(x,t+\tau_{\infty})$ for $(x,t)\in\real^N\times(0,\infty)$, where $U_*$ is the unique compactly supported self-similar solution to Eq. \eqref{eq.sing} introduced in \eqref{SSS}. In the self-similar variables \eqref{SSvar}, this upper bounds writes
$$
v(y,s)\leq \left(\frac{t+\tau_{\infty}}{t}\right)^{\alpha}f_*\left(|y|\left(\frac{t}{t+\tau_{\infty}}\right)^{\beta}\right),
$$
where $f_*$ is the self-similar profile of $U_*$ according to \eqref{SSS}, or, equivalently,
\begin{equation}\label{upper}
v(y,s)\leq \overline{U}_*(y,s):=\big(1+\tau_{\infty}e^{-s}\big)^{\alpha}f_*\big(|y|(1+\tau_{\infty}e^{-s})^{-\beta}\big), \quad (y,s)\in\real^N\times(0,\infty),
\end{equation}
which also gives that the support of $v$ is localized; more precisely
\begin{equation}\label{interm7}
{\rm supp}\,v(s)\subset B(0,\zeta_0), \quad \zeta_0:=(1+\tau_{\infty})^{\beta}\xi_0, \quad s\in(0,\infty),
\end{equation}
where $\xi_0$ is the edge of the support of $f_*$ defined in \eqref{edge.supp}. On the other hand, we have shown in Proposition \ref{prop.subs} that the subsolution $W_*$ constructed throughout this section satisfies $W_*(x,t)\leq u(x,t+t_0)$ for any $t\geq1/\lambda_*$ and $x\in\real^N$; this is equivalent to $W_*(x,t-t_0)\leq u(x,t)$ for any $t\geq t_0+1/\lambda_*$ and $x\in\real^N$. Translating this bound into the self-similar variables \eqref{SSvar} reads
$$
v(y,s)\geq\left(\frac{t-t_0}{t}\right)^{\alpha}g_*\left(|y|\left(\frac{t-t_0}{t}\right)^{-\beta}\right),
$$
where $g_*$ is the profile of the subsolution $W_*$, defined as
$$
g_*(\xi)=\lambda_*^{1/(m-1)}f(\lambda_*^{-\beta}\xi),
$$
with $f$ being the profile of the subsolution introduced in Lemma \ref{lem.subs3} and $\lambda_*\in(0,1)$ defined in the proof of Proposition \ref{prop.subs}. The previous estimate writes equivalently as
\begin{equation}\label{lower}
\begin{split}
v(y,s)\geq \overline{W}_*(y,s)&:=\big(1-t_{0}e^{-s}\big)^{\alpha}g_*\big(|y|(1-t_0e^{-s})^{-\beta}\big), \quad (y,s)\in\real^N\times(s_0,\infty),\\ &s_0:=\ln\left(t_0+\frac{1}{\lambda_*}\right).
\end{split}
\end{equation}
Let us now check the three hypothesis of the S-theorem, for which we consider the functional space
$$
X:=\{f\in L^{\infty}(\real^N)\cap C(\real^N): \overline{W}_*(y,s_0)\leq f(y)\leq\overline{U}_*(y,0), \ y\in\real^N\}
$$
and the class $\mathcal{S}$ of solutions to Eq. \eqref{eq.reg.SS} with values in $X$ at any time $s\geq s_0$. Observe that \eqref{upper}, \eqref{lower}, the monotonicity with respect to $s$ of $\overline{U}_*$ (which follows readily from the fact that $f_*$ is a decreasing profile on its support) and the evolution with respect to $s$ of the amplitude and support of $\overline{W}_*$, entail that $v(s)\in X$ for any $v$ solution to Eq. \eqref{eq.reg.SS} with compactly supported data and $s\geq s_0$, that is, $v\in\mathcal{S}$.

Since $\sigma>\max\{-N,-2\}$ and $|y|\leq\zeta_0$ by \eqref{interm7}, we deduce that
$$
(e^{-\beta s}+|y|)^{\sigma}v^p<|y|^{\sigma}v^p\in L^q(\real^N), \quad 1\leq q<\frac{N}{|\sigma|}, \quad s>0,
$$
thus we can choose $q\in(N/2,N/2|\sigma|)$ such that $N/2q<1$ and the conditions in \cite[Section 2]{DiB83} are then easily satisfied for Eq. \eqref{eq.reg.SS}. We thus deduce the compactness of the orbits $\{v(s)\}_{s>0}$ following from \cite[Proposition 1]{DiB83}, which confirms the hypothesis $(H_1)$ of the S-theorem. Moreover, if $\{v(s+s_j)\}_{j\geq1}$ is a uniformly convergent subsequence of functions to another function $v_*\in C([0,\infty);X)$ as $j\to\infty$, then
$$
(e^{-\beta (s+s_j)}+|y|)^{\sigma}v^p(y,s+s_j)\to |y|^{\sigma}v_*^p(y,s), \quad {\rm as} \ j\to\infty,
$$
and the upper bound \eqref{upper} together with the integrability of $\overline{U}_*$ and the dominated convergence theorem imply that the limit $v_*$ is a solution to Eq. \eqref{eq.sing.SS}, which means that the hypothesis $(H_2)$ of the S-theorem is satisfied. Finally, it follows from \cite[Theorem 2.11]{IL24} and its proof that the $\omega$-limit set of the orbits of solutions to Eq. \eqref{eq.sing.SS} belonging to $C([0,\infty);X)$ is the singleton $\{f_*\}$, thus, the hypothesis $(H_3)$ in the S-theorem is trivially fulfilled. By applying the S-theorem, we conclude that the $\omega$-limit set of orbits of solutions $u\in C([0,\infty);X)$ to Eq. \eqref{eq.reg.SS} is contained in the $\omega$-limit set of orbits of solutions $u\in C([0,\infty);X)$ to Eq. \eqref{eq.sing.SS}, that is, in $\{f_*\}$. The proof is completed by undoing the change of variable \eqref{SSvar}.
\end{proof}

\section{Proofs of the main results for general weights $\varrho(x)$}\label{sec.gen}

Throughout this section, we fix $\varrho$ as in \eqref{cond.varrho} and $u_0$ as in \eqref{icond}. We deduce from the positivity, continuity and boundedness of $\varrho$ that there are positive constants $c_1$, $c_2$ such that
\begin{equation}\label{equiv}
c_1(1+|x|)^{\sigma}\leq\varrho(x)\leq c_2(1+|x|)^{\sigma}\leq c_2|x|^{\sigma}, \quad x\in\real^N.
\end{equation}
These bounds are fundamental for the proofs.
\begin{proof}[Proof of Theorem \ref{th.fsp}]
Since $\varrho$ is bounded and continuous, the local well posedness follows from the classical theory (see for example \cite[Chapter V.8]{LSU}). Moreover, the inequalities \eqref{equiv} ensure that solutions to \eqref{eq.sing.c} with $c=c_2$ are supersolutions to Eq. \eqref{eq.gen} with $\varrho$ as above. Global existence and finite speed of propagation follow thus from the rescaling in Lemma \ref{lem.subs1} and the comparison principle given in \cite[Proposition 2.2]{Su02}.
\end{proof}
We next prove the large time behavior, completing thus our analysis. This proof is based on the techniques already employed in the study of Eq. \eqref{eq.reg} in the previous section, thus we will skip some details.
\begin{proof}[Proof of Theorem \ref{th.asympt}]
In the notation of Lemma \ref{lem.subs1}, the inequalities \eqref{equiv} and the comparison principle ensure on the one hand that there is $\tau_{\infty}>0$ such that
\begin{equation}\label{upper.gen}
u(x,t)\leq (U_*)_{c_2}(x,t+\tau_{\infty})=\lambda_{c_2}^{1/(m-1)+\alpha}(t+\tau_{\infty})^{\alpha}f_{*}(|x|\lambda_{c_2}^{-\beta}(t+\tau_{\infty})^{-\beta}),
\end{equation}
for any $(x,t)\in\real^N\times(0,\infty)$. On the other hand, we infer from \eqref{equiv}, Lemma \ref{lem.subs1}, Proposition \ref{prop.subs} and the comparison principle that
\begin{equation}\label{lower.gen}
u(x,t+t_0)\geq (W_*)_{c_1}(x,t)=\lambda_{c_1}^{1/(m-1)+\alpha}W_{*}(x,\lambda_{c_1}t), \quad (x,t)\in\real^N\times\left(\frac{1}{\lambda_*},\infty\right),
\end{equation}
where $\lambda_*$ is chosen as in Proposition \ref{prop.subs}. Notice that the bounds \eqref{upper.gen} and \eqref{lower.gen} already prove that the grow-up rate as $t\to\infty$ of the solution to the Cauchy problem \eqref{eq.gen}-\eqref{ic} is $t^{\alpha}$. We thus pass again to self-similar variables \eqref{SSvar}. Since $u_t$ and $\Delta u^m$ change in the same way as in Section \ref{sec.reg}, the only variation appears in the rescaling of the source term in \eqref{eq.gen}. We have
\begin{equation*}
\begin{split}
\varrho(x)u^p(x,t)&=t^{p\alpha}\varrho(t^{\beta}y)v^p(y,s)=t^{p\alpha}(1+t^{\beta}|y|)^{\sigma}\frac{\varrho(t^{\beta}y)}{(1+t^{\beta}|y|)^{\sigma}}v^p(y,s)\\
&=t^{p\alpha+\beta\sigma}(e^{-s\beta}+|y|)^{\sigma}\frac{\varrho(e^{s\beta}y)}{(1+e^{s\beta}|y|)^{\sigma}}v^p(y,s),
\end{split}
\end{equation*}
whence Eq. \eqref{eq.gen} writes in the self-similar variables $(y,s)$ defined in \eqref{SSvar} as follows:
\begin{equation*}
\partial_sv=\Delta v^m+\beta y\cdot\nabla v-\alpha v+(e^{-\beta s}+|y|)^{\sigma}\frac{\varrho(ye^{\beta s})}{(1+|y|e^{\beta s})^{\sigma}}v^p(y,s), \quad (y,s)\in\real^N\times(0,\infty).
\end{equation*}
Assuming that $u_0$ is compactly supported, the estimates \eqref{upper.gen} and \eqref{lower.gen} are readily transformed in variables $(y,s)$ into very similar bounds as \eqref{upper} and \eqref{lower} (varying by a rescaling constant related to $\lambda_{c_i}$, $i=1,2$). This allows us to apply once more the S-theorem in the same way as in the final part of Section \ref{sec.reg}, taking into account that \eqref{cond.varrho} implies that
$$
\lim\limits_{s\to\infty}(e^{-\beta s}+|y|)^{\sigma}\frac{\varrho(ye^{\beta s})}{(1+|y|e^{\beta s})^{\sigma}}=A|y|^{\sigma}.
$$
Thus, the asymptotic simplification now leads to the autonomous (with respect to the time variable $s$) equation
\begin{equation}\label{eq.gen.limit}
\partial_sv=\Delta v^m+\beta y\cdot\nabla v-\alpha v+A|y|^{\sigma}v^p, \quad (y,s)\in\real^N\times(0,\infty).
\end{equation}
The rescaling in Lemma \ref{lem.subs1} gives that the solutions to Eq. \eqref{eq.gen.limit} are in a one-to-one correspondence to the solutions to Eq. \eqref{eq.sing}, and we deduce from \cite[Theorem 2.11]{IL24} that the $\omega$-limit set of solutions to Eq. \eqref{eq.gen.limit} is the singleton
$$
\{A^{1/(m-p)}f_*\big(A^{(m-1)/(m-p)}\xi\big)\}.
$$
An application of the S-theorem similar to the one at the end of Section \ref{sec.reg} completes the proof.
\end{proof}

\bigskip

\noindent \textbf{Acknowledgements} R. G. I. and A. S. are partially supported by the Grants PID2020-115273GB-I00 and RED2022-134301-T (Spain). M. L. is partially supported by the Grant PID2022-136589NB-I00, all grants funded by MCIN/AEI/10.13039/501100011033.

\bigskip

\noindent \textbf{Data availability} Our manuscript has no associated data.

\bigskip

\noindent \textbf{Conflict of interest} The authors declare that there is no conflict of interest.

\bibliographystyle{plain}

\end{document}